\documentclass[a4,10pt]{article}

\title{A note on evaluations of multiple zeta values}
\author{Shuichi Muneta}
\date{}

\usepackage{amsmath,amsthm,amssymb}

\font\xiirm=cmr12 \font\xrm=wncyr10
\newcommand{\sh}{\xrm \mbox{sh}\, \xiirm}

\newcommand{\sw}{\,\widetilde{\sh}}

\newtheorem{thm}{Theorem}[]

\newtheorem{prop}[thm]{Proposition}

\begin{document}
\maketitle

\begin{abstract}
Multiple zeta values (MZVs) with certain repeated arguments 
or certain sums of cyclically generated MZVs 
are evaluated as rational multiple of powers of $\pi^2$. 
In this paper, we give a short and simple proof 
of the remarkable evaluations of MZVs 
established by D. Borman and D. M. Bradley. 
\end{abstract}

\section{Introduction}
The multiple zeta value (MZV) is defined by the convergent series
\[\zeta( k_{1}, k_{2}, \ldots, k_{n} )
:=\sum_{ m_{1} > m_{2} > \cdots > m_{n} >0 }
\frac{1}{ m_{1}^{k_{1}} m_{2}^{k_{2}} \cdots m_{n}^{k_{n}} },\]
where $k_{1},k_{2},\ldots,k_{n}$ are positive integers and $k_{1}\geq 2$. 
The remarkable property of MZVs is that MZVs are evaluated 
for some special arguments as rational multiple of powers of $\pi^2$. 
For example, the following evaluations were proven 
by many authors ([BBB], [H1], [Z]): 
\[\zeta( \{2\}_{m} ) = \frac{ \pi^{2m} }{ (2m+1)! } 
\quad ( m \in \mathbb{Z}_{>0} ) \]
where $\{ 2 \}_{m}$ denotes the $m$-tuple $(2,2,\ldots,2)$. 
In [Z], D. Zagier conjectured the following evaluations: 
\[\zeta( \{3,1\}_{n} ) = \frac{ 2 \pi^{4n} }{ (4n+2)! } 
\quad ( n \in \mathbb{Z}_{>0}). \]
These evaluations were proved by J. M. Borwein, D. M. Bradley, 
D. J. Broadhurst and P. Lison$\mathrm{ \check{e} }$k ([BBBL1], [BBBL2]). 
In addition, D. Bowman and D. M. Bradley proved the following theorem 
which contained these results: 
\begin{thm}[{[BB]}]
\label{thm1}
For non-negative integers $m$, $n$, we have 
\begin{align*}
\lefteqn{ \sum_{\begin{subarray}{c} j_{0} + j_{1} + \cdots + j_{2n} = m \\ 
j_{0}, j_{1}, \ldots,j_{2n} \geq 0 \end{subarray} } 
\zeta(\{2\}_{j_{0}}, 3, \{2\}_{j_{1}}, 1, \{2\}_{j_{2}}, 
\ldots, \{2\}_{j_{2n-2}}, 3, \{2\}_{j_{2n-1}}, 1, \{2\}_{j_{2n}}) } \\
& \qquad\qquad = \binom{m+2n}{m} \frac{ \pi^{2m+4n} }{ (2n+1) \cdot (2m+4n+1)! }. 
\qquad\qquad\qquad\qquad
\end{align*}
\end{thm}
In this article, we provide a short and simple proof of Theorem \ref{thm1} which 
refines the proof of Theorem 5.1 in [BB]. 

\section{Algebraic setup}
We summarize the algebraic setup of MZVs introduced by Hoffman (cf. [H2], [IKZ]). 
Let $ \mathfrak{H} = \mathbb{Q} \left\langle x,y \right\rangle$ be 
the noncommutative polynomial ring in two indeterminates $x$, $y$ and 
$\mathfrak{H}^1$ and $\mathfrak{H}^0$ its subrings 
$\mathbb{Q} + \mathfrak{H}y$ and $\mathbb{Q} + x \mathfrak{H} y$. 
We set $z_{k} = x^{k-1} y$ $(k = 1,2,3,\ldots)$. Then $\mathfrak{H}^1$ 
is freely generated by $\{z_{k}\}_{k \geq 1}$. 

We define the $\mathbb{Q}$-linear map (called evaluation map) 
$Z:\mathfrak{H}^0 \longrightarrow \mathbb{R}$ by
\[Z(1)=1 \;\; \mathrm{and} \;\; 
Z( z_{k_{1}} z_{k_{2}} \cdots z_{k_{n}}) = \zeta( k_{1}, k_{2}, \ldots, k_{n}).\]
We next define the shuffle product $\sh$ on $\mathfrak{H}$ inductively by
\begin{eqnarray*}
1 \sh w &=& w \sh 1 \;=\; w , \\
u_{1} w_{1} \sh u_{2} w_{2} 
&=& u_{1} (w_{1} \sh u_{2} w_{2}) + u_{2} (u_{1} w_{1} \sh w_{2})  
\end{eqnarray*}
($u_{1},u_{2} \in \{ x,y \}$ and $w$, $w_{1}$, $w_{2}$ are words in $\mathfrak{H}$), 
together with $\mathbb{Q}$-bilinearity. 
The shuffle product $\sh$ is commutative and associative. 
For this product, we have 
\[ Z(w_{1} \sh w_{2}) = Z(w_{1}) Z(w_{2}) \]
for any $w_{1}, w_{2} \in \mathfrak{H}^0$. 

We also define the shuffle product $\sw$ 
on $\mathbb{Q} \left\langle z_{1},z_{2},\ldots \right\rangle$ inductively by
\begin{eqnarray*}
1 \sw w &=& w \sw 1 \;=\; w , \\
u_{1} w_{1} \sw u_{2} w_{2} 
&=& u_{1} (w_{1} \sw u_{2} w_{2}) + u_{2} (u_{1} w_{1} \sw w_{2})  
\end{eqnarray*}
($u_{1},u_{2} \in \{ z_{k} \}_{k \geq 1}$ 
and $w$, $w_{1}$, $w_{2}$ are words in 
$\mathbb{Q} \left\langle z_{1},z_{2},\ldots \right\rangle$), 
together with $\mathbb{Q}$-bilinearity.
For example, we have 
\begin{align*}
z_{m} \sw z_{n} &= z_{m} z_{n} + z_{n} z_{m}, \\
z_{m} \sw z_{n} z_{l} 
&= z_{m} z_{n} z_{l} + z_{n} z_{m} z_{l} + z_{n} z_{l} z_{m}. 
\end{align*}
Then Theorem \ref{thm1} can be restated as follows: 
\[Z \left( z_{2}^m \sw (z_{3} z_{1})^n  \right) 
= \binom{m+2n}{m} \frac{ \pi^{2m+4n} }{ (2n+1) \cdot (2m+4n+1)! } 
\quad \big( m,n \in \mathbb{Z}_{\geq 0} \big). \]

\section{Proof of Theorem \ref{thm1}}
We restate Proposition 4.1 and Proposition 4.2 of [BB] by using $\sw$ 
and prove them by induction. 
\begin{prop}
\label{prop2}
For integers $n$, $N$ which satisfy $0 \leq n \leq N$, we have 
\begin{align}
z_{2}^n \sh z_{2}^N 
&= \sum_{k=0}^{n} 4^k \binom{N+n-2k}{n-k} 
\left\{ z_{2}^{N+n-2k} \sw ( z_{3} z_{1} )^k \right\}, 
\label{eq:1} \\
z_{1} z_{2}^n \sh z_{1} z_{2}^N 
&= 2 \sum_{k=0}^{n} 4^k \binom{N+n-2k}{n-k} 
z_{1} \left\{ z_{2}^{N+n-2k} \sw z_{1} ( z_{3} z_{1} )^k \right\}. 
\label{eq:2}
\end{align}
\end{prop}
\begin{proof}
We prove identities (\ref{eq:1}) and (\ref{eq:2}) 
simultaneously by induction on $n$. 
[Step 1] The case $n=0$ of (\ref{eq:1}) is clear. 
We can easily prove the case $n=0$ of (\ref{eq:2}) by induction on $N$. 
[Step 2] Suppose that (\ref{eq:1}) and (\ref{eq:2}) have been proven for $n-1$. 
We prove (\ref{eq:1}) for $n$ by induction on $N$. 
\begin{align*}
z_{2}^n \sh z_{2}^n 
&= 2 xy \{ (xy)^{n-1} \sh (xy)^n \} 
+ 2 x^2 \{ y (xy)^{n-1} \sh y (xy)^{n-1} \} 
\displaybreak[0]\\
&= 2 \sum_{k=0}^{n-1} 4^k \binom{2n-1-2k}{n-1-k} z_{2} 
\{ z_{2}^{2n-1-2k} \sw ( z_{3} z_{1} )^k \}   \\
& \quad + \sum_{k=0}^{n-1} 4^{k+1} \binom{2n-2-2k}{n-1-k} z_{3}
\{ z_{2}^{2n-2-2k} \sw z_{1} ( z_{3} z_{1} )^k \} 
\displaybreak[0]\\
&= \sum_{k=0}^{n-1} 4^k \binom{2n-2k}{n-k} z_{2} 
\{ z_{2}^{2n-1-2k} \sw ( z_{3} z_{1} )^k \} \\
& \quad + \sum_{k=1}^{n} 4^k \binom{2n-2k}{n-k} z_{3} 
\{ z_{2}^{2n-2k} \sw z_{1} ( z_{3} z_{1} )^{k-1} \} 
\displaybreak[0]\\
&= \binom{2n}{n} z_{2}^{2n} + \sum_{k=1}^{n-1} 4^k \binom{2n-2k}{n-k} 
\{ z_{2}^{2n-2k} \sw ( z_{3} z_{1} )^k \} + 4^n ( z_{3} z_{1} )^n 
\displaybreak[0]\\
&= \sum_{k=0}^{n} 4^k \binom{2n-2k}{n-k} 
\{ z_{2}^{2n-2k} \sw ( z_{3} z_{1} )^k \}.
\end{align*}
Hence (\ref{eq:1}) is true for $N=n$. 
Suppose that the case $N-1$ of (\ref{eq:1}) has been proven. 
(We may assume that $N-1 \geq n$ in the following calculation.)
\begin{align*}
z_{2}^n \sh z_{2}^N 
&= xy \{ (xy)^{n-1} \sh (xy)^N \} 
+ 2 x^2 \{ y (xy)^{n-1} \sh y (xy)^{N-1} \} \\
&\quad + xy \{ (xy)^{n} \sh (xy)^{N-1} \} 
\displaybreak[0]\\
&= \sum_{k=0}^{n-1} 4^k \binom{N+n-1-2k}{n-1-k} z_{2} 
\{ z_{2}^{N+n-1-2k} \sw ( z_{3} z_{1} )^k \} \\
& \quad + \sum_{k=0}^{n-1} 4^{k+1} \binom{N+n-2-2k}{n-1-k} z_{3} 
\{ z_{2}^{N+n-2-2k} \sw z_{1} ( z_{3} z_{1} )^k \} \\
& \quad + \sum_{k=0}^{n} 4^k \binom{N+n-1-2k}{n-k} z_{2} 
\{ z_{2}^{N+n-1-2k} \sw ( z_{3} z_{1} )^k \} 
\displaybreak[0]\\
&= \sum_{k=0}^{n-1} 4^k \binom{N+n-2k}{n-k} z_{2} 
\{ z_{2}^{N+n-1-2k} \sw ( z_{3} z_{1} )^k \} \\
& \quad + \sum_{k=1}^{n} 4^k \binom{N+n-2k}{n-k} z_{3} 
\{ z_{2}^{N+n-2k} \sw z_{1} ( z_{3} z_{1} )^{k-1} \} \\
&\quad + 4^n z_{2} \{ z_{2}^{N-n-1} \sw ( z_{3} z_{1} )^n \} 
\displaybreak[0]\\
&= \binom{N+n}{n} z_{2}^{N+n} 
+ \sum_{k=1}^{n-1} 4^k \binom{N+n-2k}{n-k} 
\{ z_{2}^{N+n-2k} \sw ( z_{3} z_{1} )^k \} \\
&\quad + 4^n \{ z_{2}^{N-n} \sw ( z_{3} z_{1} )^n \} 
\displaybreak[0]\\
&= \sum_{k=0}^{n} 4^k \binom{N+n-2k}{n-k} 
\{ z_{2}^{N+n-2k} \sw ( z_{3} z_{1} )^k \}. 
\end{align*}
Hence (\ref{eq:1}) is true for $N$. 
We can prove (\ref{eq:2}) for $n$ by induction on $N$ 
with using (\ref{eq:1}) for $n$. 
\end{proof}

Before proceeding the proof of Theorem \ref{thm1}, 
we prove a key identity. 
Comparing coefficients of 
$ (x+1)^{2m+4n+2} = (x^2 + 2x + 1)^{m+2n+1}$, we have 
\[ \binom{2m+4n+2}{2n+1} = \sum_{k=0}^{n} 2^{2k+1} 
\frac{ (m+2n+1)! }{ (n-k)! (2k+1)! (m+n-k)! }. \]
We can transform this identity as follows: 
\begin{align}
\frac{1}{(2n+1)!} \frac{1}{(2m+2n+1)!}  
&= \sum_{k=0}^{n} 4^k \binom{m+2n-2k}{n-k} \binom{m+2n}{2k} 
\frac{ 1 }{ (2k+1) \cdot (2m+4n+1)! }. 
\label{eq:3}
\end{align}

\begin{proof}[Proof of Theorem \ref{thm1}]
We prove Theorem \ref{thm1} by induction on $n$. 
The case $n=0$ is well known as has been mentioned in Section 1. 
Suppose that the assertion has been proven up to $n-1$. 
Putting $N=m+n$ in (\ref{eq:1}), we have 
\begin{align*}
\lefteqn{ 4^n Z \left( z_{2}^m \sw (z_{3} z_{1})^n \right) } \\
&= \frac{ \pi^{2n} }{ (2n+1)! } \frac{ \pi^{2m+2n} }{ (2m+2n+1)! }
- \sum_{k=0}^{n-1} 4^k \binom{m+2n-2k}{n-k} \binom{m+2n}{2k} 
\frac{ \pi^{2m+4n} }{ (2k+1) \cdot (2m+4n+1)! } 
\displaybreak[0]\\
&\stackrel{(\ref{eq:3})}{=} 
4^n \binom{m+2n}{m} \frac{ \pi^{2m+4n} }{ (2n+1) \cdot (2m+4n+1)! }. 
\end{align*}
This completes the proof of Theorem \ref{thm1}. 
\end{proof}

\vspace{5mm}
\textsc{Graduate School of Mathematics, Kyushu University}

\textsc{Fukuoka 812-8581, Japan}

\textit{E-mail address}: muneta@math.kyushu-u.ac.jp

\end{document}